\documentclass[12pt]{article}
\usepackage{amsmath}
\usepackage{amssymb}
\usepackage{amsthm}
\usepackage{pgf,tikz,pgfplots}
\pgfplotsset{compat=1.14}
\usepackage{mathrsfs}
\usepackage{comment}
\usepackage{enumitem}

\theoremstyle{plain}
\newtheorem{thm}{Theorem}
\newtheorem{prop}[thm]{Proposition}
\newtheorem{lemma}[thm]{Lemma}
\newtheorem{defn}[thm]{Definition}
\newtheorem{cor}[thm]{Corollary}

\newtheorem{obser}[thm]{Observation}
\newtheorem{question}[thm]{Question}
\newtheorem*{thmbezcisla}{Theorem}

\newtheorem{claim}{Claim}

\theoremstyle{remark}
\newtheorem*{rem}{Remark}
\newtheorem*{notat}{Notation}

\newcommand{\R}{\mathbb{R}}
\newcommand{\N}{\mathbb{N}}

\DeclareMathOperator{\cl}{Cl}
\let\int\relax
\DeclareMathOperator{\int}{Int}
\DeclareMathOperator{\diam}{diam}

\DeclareMathOperator{\bd}{\partial}

\newcommand{\F}{\mathcal{F}}

\newcommand{\lifted}[1]{\overline{#1}}
\newcommand{\periodic}[1]{\bar{#1}}

\def\keywords#1{\par\medskip
\noindent\textbf{Keywords.} #1}

\begin{document}
\makeatletter

\title{Topological fractals revisited} 
\author{Klára Karasová
\footnote{Orcid: 0000-0002-9019-443X.
The paper was supported by Charles University project PRIMUS/21/SCI/014 and by the grant GA UK No. 129024.}
\ and
Benjamin Vejnar
\footnote{Orcid: 0000-0002-2833-5385. 
This paper was supported by the grant GACR 24-10705S.}
\\ \\
Faculty of Mathematics and Physics \\
Charles University \\
Prague, Czechia\\ \\
}

\maketitle

\begin{abstract}
We prove that every Peano continuum with uncountably many local cut points is a topological fractal. This extends some recent results and gives a partial answer to a conjecture by Hata. We also discuss the number of maps which are sufficient for witnessing the structure of a topological fractal.
\end{abstract}


\section{Introduction}

In this paper we are going to deal with so called topological fractals. Nevertheless, to motivate the notion, we start with the attractors of IFS.
Let $(X, d)$ be a metric space. A map $f:X\to X$ is called a \emph{contraction} if there is a constant $c<1$ such that $d(f(x), f(y))< c\cdot d(x,y)$ for $x\neq y$. In the case $c=1$ we are talking about a \emph{weak contraction}.
A set of finitely many contractions $\{f_1,\,f_2,\,\dots,\,f_n\}$ on a fixed complete metric space $X$ is called an \emph{iterated function system} (IFS). As an application of the Banach fixed point theorem, Hutchinson observed in 1981 (see \cite{Hutchinson} or \cite[Theorem 4.1.3]{fractalgeometry}) that for every IFS with $X$ complete there exists a unique nonempty compact set $K\subseteq X$ for which 
\begin{align}\label{pokryvani}
    K=f_1(K)\cup\dots\cup f_n(K).
\end{align}
Every $K$ which can be obtained in this way will be called a \emph{metric fractal}.
The most basic examples of metric fractals are the interval $[0,1]$, the Cantor set or the Sierpiński triangle. The case when the IFS consists of similarities in an Euclidean space is of special interest. The circle is an example of a metric fractal without a witnessing IFS consisting of similarities.

It was proven in \cite{fractalgeometry} that every metric fractal is of a finite topological dimension and thus there are continua which are not metric fractals.
It was noted by Hata \cite{Hata} that every connected metric fractal is a Peano continuum (see Theorem \ref{coverings} for more details). In the same paper he asked whether every finite--dimensional Peano continuum is a metric fractal. This question was answered negatively, in fact there exists a plane 1--dimensional Peano continuum which is not even homeomorphic to any IFS attractor \cite{onedimensional}.

Further generalizations of metric fractals are naturally obtained by considering only weak contractions instead of contractions. In this way we get the notion of a \emph{weak IFS attractor}.
It was shown in \cite{weakattractor} that metrizable weak attractors are exactly so called \emph{topological fractals}:

\begin{notat}
Let $X$ be a set and $\F$ be a set of self-maps of $X$. We denote $\F^0:= \{id\}$ and for any $n>0$, we denote by $\F^n$ the set of maps $\{f_1\circ\dots\circ f_n;\,f_i \in \F \text{ for }1\leq i \leq n\}$.
\end{notat}

\begin{defn}
Let $X$ be a metric space and $\F$ be a set of self-maps of $X$. We say that $\F$ is topologically contractive if for every $\varepsilon >0$ there exists $k \in \N$ such that $\diam f(X)<\varepsilon$ for every $f \in \F^k$.
\end{defn}

Note that the notion of a topological contractivity can be easily reformulated using open covers instead of the metric, and thus this property only depends on the topology of the space.

\begin{defn}
Let $X$ be a compact metric space. We say that $X$ is a topological fractal if there exists a topologically contractive finite set $\F$ of continuous self-maps of $X$ satisfying $ \bigcup\{f(X);f\in \F\}=X$.
In that case, the maps in $\F$ are called witnessing maps and the smallest cardinality of a family $\mathcal F$ of witnessing maps is called the witnessing number.
\end{defn}

Similarly to the case of metric fractals, every connected topological fractal has the property S and thus is a Peano continuum by Theorem \ref{coverings}. Thus among continua, only Peano continua might happen to be topological fractals. However the question, whether every Peano continuum is a topological fractal, is still open since 1985 when Hata posted it in \cite{Hata}.
This question is also mentioned in the Lviv Scottish Book (page 3, Volume 0).

Nevertheless, some sufficient conditions under which a Peano continuum is a topological fractal have been found in the meantime. To begin, every Peano continuum containing a free arc, i.e. containing an arc as a subset with nonempty interior, is a topological fractal \cite{freearc1, freearc2}. 
As a consequence of this result, the space $[0,1]^\N \cup [0,1]$ with the points $(0,0,\dots)$ and $0$ identified, is an infinite--dimensional connected topological fractal. Thus this space is in particular a topological fractal that is not a metric fractal. 

However, Peano continua are topological fractals under much weaker conditions than containing a free arc. It is proved in \cite{regenerating} that every Peano continuum containing a self-regenerating subcontinuum with a nonempty interior is a topological fractal itself. The paper contains some examples of self-regenerating continua, including the arc.
In this paper we present a new sufficent condition under which a Peano continuum is a topological fractal. Namely we prove in Theorem \ref{thmlocalcutpoint}:

\begin{thmbezcisla}
Every Peano continuum with uncountably many local cut points is a topological fractal.
\end{thmbezcisla}

This Theorem strengthens the results that every Peano continuum containing a free arc is a topological fractal \cite{freearc1, freearc2}, as all points in the interior of a free arc (the interior is taken with respect to the whole space) are local cut points or even cut points of the whole space. Yet, the Theorem is independent of the result of M. Nowak \cite{regenerating}. While she works with sets that have nonempty interior, the uncountable set of local cut points might be nowhere dense.

Beside the sufficient condition under which a Peano continuum is a topological fractal, we obtain that the number three is an upper bound on the witnessing number for Peano continua with uncountably many local cut points. In the case of Peano continua with uncountably many cut points we further reduce the upper bound to two, which is optimal.

\section{Preliminaries} 

Let $X$ be a metrizable space. When $A$ is a subset of $X$, we denote by $\int(A)$ the interior of $A$, by $\cl(A)$ the closure of $A$ and by $\bd(A)$ the boundary of $A$. We write $\diam(A)$ for the diameter of $A$ if some metric giving the topology on $X$ is fixed. We denote positive integers by $\N$ a non-negative integers by $\omega$. We use $\N$ when working with individual sequences, and $\omega$ when working with sets of sequences.

We say that $X$ (or a subset of $X$) is \emph{connected}, if it cannot be written as a union of two proper disjoint open sets. Maximal connected sets are called (connected) \emph{components}. These sets are always closed and they form a partition of $X$, thus in particular $X$ is connected if and only if it consists of a single connected component, namely $X$.

Let $X$ be a (connected) topological space. A point $x \in X$ is called a \emph{cut point} if $X \setminus \{x\}$ is not connected. We say that $x \in X$ is a \emph{local cut point} of $X$ if there exists $U$ a connected neighborhood of $x$ such that $U \setminus \{x\}$ is not connected (i.e. $x$ is a cut point of $U$).
Clearly every cut point of a connected space is also a local cut point of that space.

A \emph{continuum} is a nonempty connected compact metrizable space. For the fundamentals of continuum theory we refer to the book by S. Nadler \cite{nadler}.
We say that a space $X$ is \emph{locally connected}, if the set of all connected open subsets of $X$ forms a basis of the topology of $X$. A locally connected continuum is called \emph{a Peano continuum}. We say that a metric space $X$ has the property $S$ if for every $\varepsilon>0$ there exists a finite cover of $X$ by connected sets with diameter less than $\varepsilon$.
In what follows, we recall several properties of Peano continua. For more details on Peano continua see \cite[Chapter 8]{nadler}.

\begin{thm}[{\cite[Theorems 8.4, 8.18, 8.19]{nadler}}] \label{coverings}
	Let $(X, d)$ be a metric space. Then the following conditions are equivalent:
	\begin{enumerate}[noitemsep]
		\item\label{peano1} $X$ is a Peano continuum.
		\item\label{peano2} $X$ is a continuum with property $S$.
		\item\label{peano5} There exists a continuous onto map $f:I\to X$.
		\item\label{peano6} For every Peano continuum $Y$, $a \ne b \in Y$ and $c, d \in X$ there exists $f:Y \to X$ continuous, onto and satisfying $f(a) = c$, $f(b) = d$.
	\end{enumerate}
\end{thm}

An \emph{arc} in a space $X$ is a continuous one--to--one map $f:I \to X$, respectively its image. If $f(0)=x$ and $f(1)=y$ we say that $f$ is an arc from $x$ to $y$. We say that $S \subseteq X$ is arcwise connected if for every $x \ne y\in S$ there exists an arc in $S$ from $x$ to $y$.

\begin{thm}\cite[Theorem 8.26]{nadler} \label{arcwiseconnected}
In Peano continua,  every open and connected set is arcwise connected.
\end{thm}

\begin{obser}
    \label{opencomponents}
    In Peano continua, components of open sets are open.
\end{obser}

Let us start with the following useful observation. A similar result was proved indepently in \cite{regenerating}. For the sake of completeness we include a simple proof.

\begin{obser} \label{fractalmaps}
	Let $X$ be a compact metric space, $m,n \in \N$ and $f_1,\,\dots, f_n,\,g_1,\,\dots,\,g_m: X \to X$ be continuous maps. Suppose that $\{f_1,\,\dots,\,f_n\}$ is topologically contractive and that for every $f \in \bigcup_{i=0}^{\infty} \{f_1,\,\dots,\,f_n\}^i$ and $1\leq i,\,j \leq m$ the map $g_i\circ f \circ g_j$ is constant.
	
	Then $\{f_1,\,\dots, f_n,\,g_1,\,\dots,\,g_m\}$ is topologically contractive.
\end{obser}

\begin{proof}
Denote by $d$ the metric on $X$ and let $\varepsilon >0$. Then there exists $k_1 \in \N$ such that $\diam f(X)<\varepsilon$ for every $f \in \{f_1,\,\dots,\,f_n\}^{k_1}$. We can find $\delta >0$ such that for every $f \in \bigcup_{i=0}^{k_1-1} \{f_1,\,\dots,\,f_n\}^i$ and every $1\leq j \leq m$, if $x,y \in X$, $d(x,y)<\delta$, then $d(f(g_j(x)),f(g_j(y)))<\epsilon$. By assumption, there exists $k_2 \in \N$ such that $\diam f(X)<\delta$ for every $f \in \{f_1,\,\dots,\,f_n\}^{k_2}$.
	
Let $k:= k_1 + k_2$ and $h = h_1\circ\dots\circ h_k \in \{f_1,\,\dots, f_n,\,g_1,\,\dots,\,g_m\}^k$. If there are $1 \leq i\ne j\leq k$ such that $h_i,\,h_j \in \{g_1,\,\dots,\,g_m\}$, then $h$ must be by assumption a constant map and hence $\diam h(X) = 0 <\varepsilon$.
	
Suppose that $h_i \in \{g_1,\,\dots,\,g_m\}$ for at most one $1 \leq i \leq k$. If $h_1,\,\dots ,\,h_{k_1} \in \{f_1,\,\dots, f_n\}$, then $\diam h(X) < \varepsilon$ since $h(X) = h_1\circ\dots\circ h_k(X) \subseteq h_1\circ\dots\circ h_{k_1}(X)$ and $h_1\circ\dots\circ h_{k_1} \in \{f_1,\,\dots,\,f_n\}^{k_1}$. Otherwise $h_i \in \{g_1,\,\dots,\,g_m\}$ for exactly one $1\leq i \leq k_1$ and $h_{i+1},\,\dots,\, h_k \in \{f_1,\,\dots, f_n\}$. Then $\diam(h_{i+1}\circ\dots\circ h_k(X)) < \delta$ since $k - i\geq k -k_1 = k_2$.

Thus by the definition of $\delta$, $\diam h(X) \leq \varepsilon$.
\end{proof}

\begin{obser} \label{threetotwo}
Let $X$ be a metric space and $f,\,g: X \to X$ continuous self-maps of $X$. If $\{f,\,gf,\,gg\}$ is topologically contractive, then so is $\{f,\,g\}$.
\end{obser}

\begin{proof}
Let $\varepsilon>0$, then there exists $k \in \N$ such that $\diam h(X)<\varepsilon$ for every $h \in \{f,\,gf,\,gg\}^k$. Let $h' \in \{f,\,g\}^{2k}$ be arbitrary and notice that actually $h' = h\circ r$ for some $h \in \{f,\,gf,\,gg\}^k$ and $r: X \to X$. This is true since $h'$ is, by definition, a composition of $2k$ maps from $\{f,\,g\}$ and we may regroup these maps starting from left, creating a sequence of $k$ maps from $\{f,\,gf,\,gg\}$, possibly leaving out some maps in the end. For example:
    
\[ f\circ f\circ g\circ g \circ g \circ f \circ f \circ g \circ f\circ f\circ g\circ g \circ g \circ f \circ f \circ g = \]
\[= ((f)\circ (f)\circ (g\circ g) \circ (g \circ f) \circ (f) \circ (g \circ f)\circ (f)\circ (g\circ g)) \circ (g \circ f \circ f \circ g).\]

Thus $h'(X) = hr(X) \subseteq h(X)$ and $\diam h(X) <\varepsilon$.
\end{proof}

\begin{lemma}\label{diameters}
Let $X$ be a Peano continuum and let $K_n\subseteq X$, $n\in\mathbb N$, be pairwise disjoint subcontinua whose boundaries have at most 13 points. Then $\diam K_n$ converges to zero.
\end{lemma}

\begin{proof}
Suppose for contradiction that there is $\varepsilon>0$ and an infinite set $S\subseteq \mathbb N$, such that $\diam K_n \geq\varepsilon$ for $n\in S$. 
Since the hyperspace of subcontinua of $X$ is compact, there is an infinite set $T\subseteq S$ such that $K_n, n\in T$ converges to some continuum $K\subseteq X$. Clearly, $\diam K \geq\varepsilon$, so $K$ is non-degenerate. 
Hence we can find distinct points $x_1,\dots x_{14}\in K$.
Let $L_i$ be a locally connected subcontinuum of $X$ which contains $x_i$ in its interior and assume moreover that $L_i\cap L_j=\emptyset$ for $i\neq j$.
Using the Vietoris topology, there exist $m, n \in T$, $m\neq n$, for which both $K_m$ and $K_n$ intersects every $L_i$, $i=1,\dots, 14$.

We easily get that the boundary of $K_m$ intersects each set $L_i$; indeed if
for some $i$ the boundary of $K_m$ does not intersect $L_i$, then $L_i = (L_i \setminus K_m) \cup (L_i \cap \int (K_m))$ would form a decomposition into open and nonempty sets since $\emptyset \ne L_i \cap K_n \subseteq L_i \setminus K_m$ and $\emptyset \ne L_i \cap K_m = L_i \cap \int (K_m)$, which contradicts the connectedness of $L_i$. Thus the boundary of $K_m$ intersects each $L_i$, hence the boundary consists of at least 14 points. This is a contradiction. 
\end{proof}

\begin{lemma}\label{continuous}
Let $X, Y$ be metric spaces and $X=\bigcup_{n\in\N} X_n$ where every $X_n$, $n\in\N$, is closed, $X_n\cap X_m=\emptyset$ for $n, m\geq 2, n\neq m$, $X_1\cap X_n\neq\emptyset$ for every $n\in\N$, and $\diam X_n$ tends to zero.
Let $f:X\to Y$ be a map whose restrictions to $X_n$, $n\in\N$, are continuous and $\diam f(X_n)$ converges to zero. Then $f$ is continuous.
\end{lemma}

\begin{proof}
Suppose for contradiction that $f$ is not continuous.
Thus there is a sequence $\{x_n\}$ which converges to $x$ in $X$ such that $\{f(x_n)\}$ does not converge to $f(x)$ in $Y$.
Hence there is $\varepsilon>0$ such that $f(x_n)$ is not contained in the $\varepsilon$-ball with center $f(x)$ for infinitely many $n\in\N$. By choosing a convenient subsequence of $\{x_n\}$, we may suppose that it holds for all $n\in\N$. Thus no subsequence of $f(x_n)$ converges to $f(x)$.
We distinguish these cases:
\begin{itemize}[noitemsep]
\item[a)] $\{k\in\N: \exists n\in\N: x_n\in X_k\}$ is finite,
\item[b)] $\{k\in\N: \exists n\in\N: x_n\in X_k\}$ is infinite.
\end{itemize}
By passing to a suitable subsequence we may assume one of the following cases (still having that $f(x_n)$ does not converge to $f(x)$):
\begin{itemize}[noitemsep]
\item[a')] $\exists k\in\N\forall n\in\N: x_n\in X_k$,
\item[b')] $x_n \in X_{k_n}$ for pairwise distinct $k_n\in {\N}$, $n\in\N$.
\end{itemize}

In case a') we get that $x\in X_k$ since $X_k$ is closed. Then by continuity of  $f|_{X_k}$ it follows that $f(x_n)$ converges to $f(x)$, which is a contradiction.
In case b') we choose $c_n\in X_1\cap X_{k_n}$. Since $\diam X_{k_n}$ converges to zero it follows that $c_n$ converges to $x$. Since $X_1$ is closed we get $x\in X_1$ and since $f|_{X_1}$ is continuous it follows that $f(c_n)$ converges to $f(x)$. Since $\diam f(X_{k_n})$ converges to zero it follows that $f(x_n)$ converges to $f(x)$, which is a contradiction.
\end{proof}

\begin{obser} \label{openconninPeano}
Let $X$ be a Peano continuum and $U \subseteq X$ a nonempty open connected set with finite boundary. Then $\cl(U)$ is a Peano subcontinuum of $X$.
\end{obser}

\begin{proof}
Obviously $\cl(U)$ is a continuuum, so we only need to prove that it is locally connected. Clearly $\cl(U)$ and $X\setminus U$ are closed sets, their union is the whole space $X$ and their intersection  is  locally connected  because it is finite. Thus by \cite[Exercise 8.37]{nadler} $\cl(U)$ is locally connected as well.

\end{proof}

\section{Main results}
Having all the preliminary statements verified, we are prepared to prove that containing uncountably many (local) cut points is a sufficient condition for Peano continua to be topological fractals (see Theorem \ref{thmunccutpoints} and Theorem \ref{thmlocalcutpoint}), which is the main result of this paper.

We will start by introducing some notation concerning both finite and infinite sequences of zero's and one's, as some proofs make use of the Cantor set. We proceed by introducing the so-called Menger order of a point and giving an easy observation to use later. 

Afterwards, we prove Lemma \ref{structureprop} stating that every Peano continuum $X$ with uncountably many cut points is of a specific structure that enables us to define self-maps of  $X$ naturally. In Theorem \ref{thmunccutpoints} we prove that we can find two such self-maps of $X$ that witness that $X$ is a topological fractal.

Although it is possible to prove the results for Peano continua with uncountably many local cut points using very similar approach, we have decided to make use of a semiconjugacy with the circle in this case. Thus in Lemma \ref{LemmaMonotoneMapOntoCircle} we prove that every Peano continuum $X$ with uncountably many local cut points is semiconjugated with the circle and in Lemma \ref{smallpreimages} we give some properties of such semiconjugacies. Lifting three contractions on the circle through the semiconjugacy to self-maps of $X$ gives three maps witnessing that $X$ is a topological fractal, as it is shown in the proof of Theorem \ref{thmlocalcutpoint}, which closes the section.

Let us now introduce some notation. We realize the Cantor set by $2^\omega=\{0,\,1\}^\omega$, the space of infinite sequences of $0,\,1$ with the product topology. We will denote $\periodic{0}:= (0,\,0,\,\dots),\,\periodic{1}:= (1,\,1,\,\dots) \in 2^\omega$. We denote by $2^{<\omega}$ the set of all finite sequences of $0,\,1$, and for any $s \in 2^{<\omega}$ we denote by $|s|$ its length (we do not assign to this set any further structure).

When working with finite and infinite sequences, the following two operations will be convenient. Let $s = (s_1,\,\dots,\,s_n),\,t=(t_1,\,\dots,\,t_k) \in 2^{<\omega}$ and $r=(r_1,\,r_2,\,\dots)$ $\in 2^\omega$, we define the concatenation by $s^\frown t:= (s_1,\,\dots,\,s_n,\,t_1,\,\dots,\,t_k) \in 2^{<\omega}$, respectively $s^\frown r:= (s_1,\,\dots,\,s_n,\,r_1,\,r_2,\,\dots) \in 2^\omega$.
Further, define $s^{op}:= (1-s_1,\,\dots,\,1-s_n)$, respectively $r^{op}:= (1-r_1,\,1-r_2,\,\dots)$.

\begin{defn}
    Let $X$ be a topological space and $x\in X$. The Menger order of $x$ in $X$ is the least natural number $n$ such that there is a basis of the neighborhood system of $x$ formed by sets whose boundaries are of size at most $n$, if any such natural number exists, and $\infty$ otherwise.
\end{defn}

\begin{obser}\label{Mengerorder}
    Let $X$ be a connected Hausdorff space and let $x\in X$ satisfy that $X\setminus \{x\}$ can be written as a disjoint union of $n$ nonempty open sets for some $n\in\N$. Then the Menger order of $x$ is at least $n$.
\end{obser}

\begin{proof}
    Let $X\setminus \{x\}= A_1\cup\dots\cup A_n$ be a disjoint union of nonempty open sets. Then $\cl(A_i) = A_i\cup \{x\}$ for every $1\leq i \leq n$ by connectedness of $X$. Let $U$ be an open set containing $x$ such that its complement intersects every $A_i$ and $V\subseteq U$ be any neighborhood of $x$. We will show that $\bd(V)$ intersects each $A_i$, thus implying that it is of size at least $n$.

    Fix $1\leq i \leq n$ and note that $(X\setminus V)\cap A_i$ is a nonempty proper subset of $X$ and thus its boundary is nonempty. Further, $\bd((X\setminus V)\cap A_i)\subseteq \bd(X\setminus V)\cup \bd(A_i)=\bd(V)\cup \{x\}$. Also, $\bd((X\setminus V)\cap A_i)\subseteq \cl(A_i)= A_i\cup \{x\}$. But $x\notin \bd ((X\setminus V)\cap A_i)\subseteq \cl (X\setminus V)$ since $x$ lies in the interior of $V$, hence we obtain $\emptyset \ne \bd((X\setminus V)\cap A_i)\subseteq \bd(V)\cap A_i$.
\end{proof}

\begin{prop}\label{structureprop}
Let $X$ be a Peano continuum with uncountably many cut points. Then there exist pairwise disjoint non-degenerate Peano subcontinua $L, R$, $X_s, s\in 2 ^{<\omega}$, of $X$, points $l, r, l(s), r(s)\in X$ and a homeomorphism $h:2^\omega\to C\subseteq X$ such that 
\begin{equation}
\begin{split} 
L\cap C &=\{l\}, \\
R\cap C &=\{r\}, \\
X_s\cap C &=\{l(s), r(s)\}, \\
h(\periodic 0)&=l, h(\periodic 1)=r, \\
h(s^\frown 1\periodic0)&=r(s), h(s^\frown 0\periodic1)=l(s), \\
X &= L\cup R\cup C\cup \bigcup X_s. \\
\end{split}
\end{equation}
\end{prop}

\begin{proof}
Let $\mathcal B$ be a countable base for $X$ formed by nonempty connected sets. For $U, V\in\mathcal B$ let $D(U, V)$ be the set of all cut points that separate $U$ and $V$ (i.e. the set of points $p\in X$ for which there exist open disjoint sets $Y, Z$ with $X\setminus\{p\}=Y\cup Z$ and $U\subseteq Y$, $V\subseteq Z$).
Since every cut point of $X$ lies in $D(U,V)$ for some $U, V\in\mathcal B$ and the set of all cut points is uncountable, there are some $U, V\in\mathcal B$ for which $D(U, V)$ is uncountable.

Let $E$ be any arc from $U$ to $V$, there is some by Theorem \ref{arcwiseconnected}. We immediately obtain $D:=D(U, V)\subseteq E$ since the set $U \cup E \cup V$ is connected. It also follows easily by local connectedness of $X$ that the set $D$ is closed.
Let $F$ be the collection of all $x\in D$ of Menger order at least three united with the endpoints of $E$. Then $F$ is countable by \cite[Theorem 7]{Whyburn} and the complement of any point in $D\setminus F$ consists of only two components by Observations \ref{opencomponents} and \ref{Mengerorder}.
Since $D\setminus F$ is an uncountable Borel set (in fact $G_{\delta}$), it contains a copy of the Cantor set $C$ \cite[Theorem 13.6]{Kechris}.
Let $h: 2^\omega\to C$ be an increasing homeomorphism (with respect to the lexicographical order on $2^\omega$ and the order on $C$ given by the order of $E$)  (folklore - e.g. by the General mapping theorem \cite[Theorem 7.4]{nadler}).

Let $L_s$ (resp. $R_s$) be the component of $X\setminus h(s)$, $s\in 2^{\omega}$, which contains $U$ (resp. $V$). The sets $L_s,R_s$ are open by Observation \ref{opencomponents}.

For $s \in 2^{<\omega}$ let 
\begin{align*}
    L &:=\cl (L_{\periodic0})=L_{\periodic0}\cup \{h(\periodic0)\}, \\
    R &:=\cl (R_{\periodic1})=R_{\periodic1}\cup \{h(\periodic1)\}, \\
    X_s &:=\cl (L_{s^\frown1\periodic0}\cap R_{s\frown0\periodic1})=(L_{s^\frown1\periodic0}\cap R_{s\frown0\periodic1})\cup \{h(s^\frown1\periodic0), h(s^\frown0\periodic1)\}.
\end{align*}
    
These sets are the closures of open connected sets with finite boundaries, hence by Lemma \ref{openconninPeano} these are Peano continua. Note that all continua $L,R,X_s, s\in 2^{<\omega},$ are non-degenerate since they contain non-degenerate subsets of $E$ (for $L,R$ this follows from the fact that the endpoints of $E$ do not lie in $C$). Moreover, they are mutually disjoint since for $s<t\in 2^\omega$ we have $L_s\subseteq L_t\setminus\{h(t)\}=X\setminus R_t$.
It remains to verify that $X=L\cup R\cup C \cup \bigcup X_s$.

Fix any $x\in X\setminus C$. By Theorem \ref{arcwiseconnected}, there is an arc $A$ joining $x$ and $h(\periodic{0})$. Consider the point $x'\in A\cap E$ such that no points in $A$ between $x$ and $x'$ lie in $E$ (we may have $x=x'$). We claim that $x'\notin C$: otherwise let $ W$ be a neighborhood of $x'$ with a two-element boundary that contains neither $x$, nor any of the endpoints of $E$. Thus necessarily both points of the boundary of $W$ lie in $E$ and at least one of them lies in $A$, a contradiction.
It is now easy using $x'$ to find an appropriate $Y\in \{L,R\}\cup \{X_s,s\in 2^{<\omega}\}$ such that $x\in Y$.
\end{proof}

\begin{thm}\label{thmunccutpoints}
Let $X$ be a Peano continuum with uncountably many cut points.
Then $X$ is a topological fractal with witnessing number two.
\end{thm}

\begin{proof}
Let us find Peano continua $L, R, X_s$, points $l, r, l(s), r(s)$ and a homeomorphism $h:2^\omega\to C$ as ensured by Proposition \ref{structureprop}.
To simplify the notation, we suppose that $C=2^\omega$ and $h$ is the identity.
Note that by Theorem \ref{coverings} we can map continuously every non-degenerate Peano continuum onto any other, also with prescription on some points. Hence we are able to find a map $f:X\to X$ which is continuous on each piece $R, L, C, X_s, s\in 2^{<\omega},$ such that (see Figure \ref{fig:mapf})
\begin{align*}
f(R) &= L  &     f(r) &=l, \\
f(L) &= X_\emptyset &  f(l)&=l({\emptyset}), \\
f(X_s) &= X_{0^\frown s^{op}} &   f(l(s)) &= r({0^\frown s^{op}}), \\
& & f(r(s)) &= l({0^\frown s^{op}}), s\in 2^{<\omega},\\
f(C)&= 0^\frown C & f(t)&= 0^\frown t^{op}, t\in C.  \\
\end{align*}

\begin{figure}[htp]
    \centering
    \includegraphics[width=14cm]{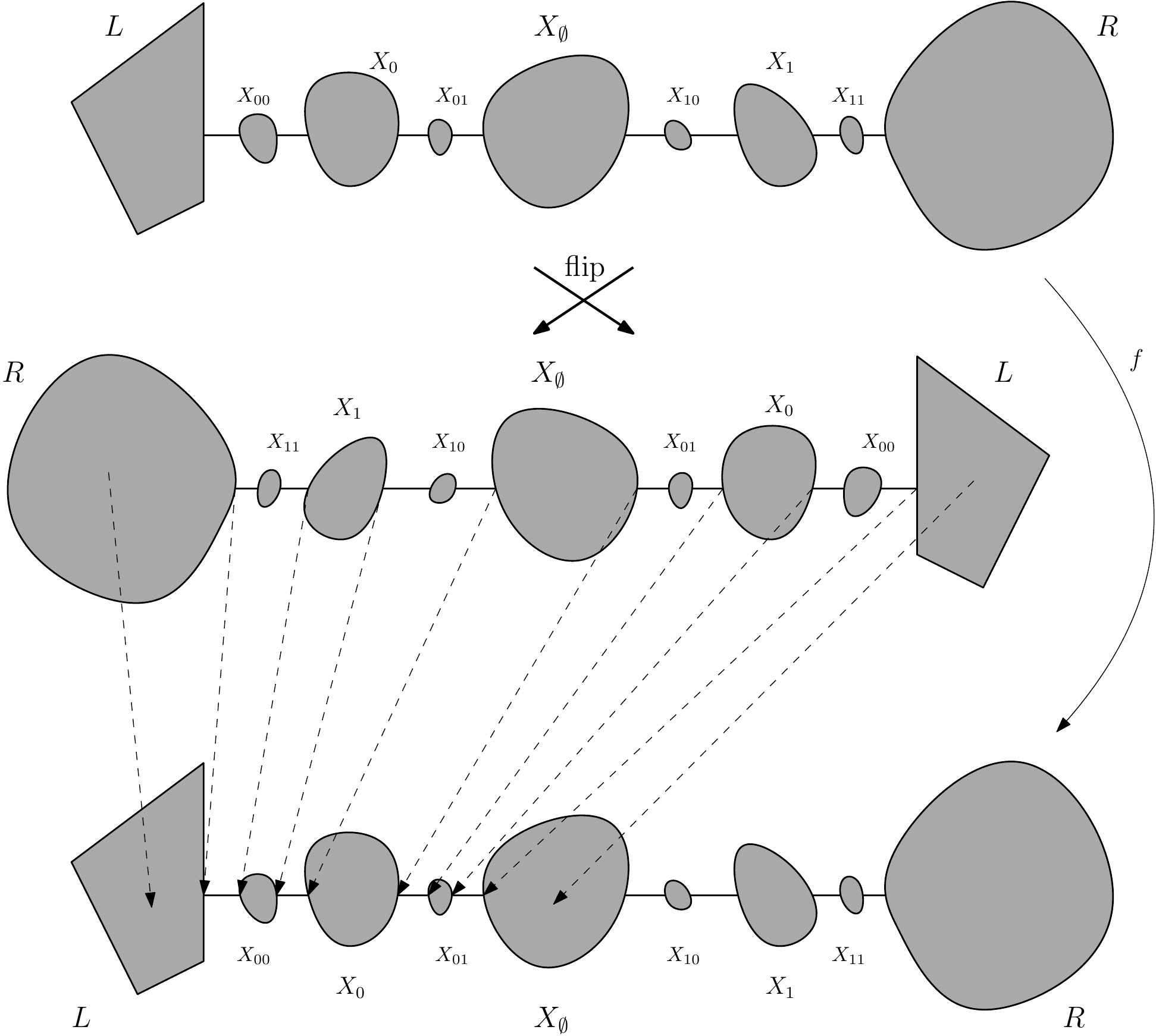}
    \caption{The map $f$}
    \label{fig:mapf}
\end{figure}


	
Let $c:2^\omega \to I$ be the well-known continuous onto map (The Cantor Stairs) and let $p: I \to R$ be any surjective continuous map satisfying $p(0) = \periodic{1} ( = r)$. Denote $q=p\circ c$. Thus $q$ maps continuously $C$ onto $R$ in such a way that $q(\periodic{0})=r$ and $q(l(s))=q(r(s))$ for every $s\in 2^{<\omega}$.


By the same reasons as above we may define a map $g:X\to X$ which is continuous on each piece $R, L, C, X_s, s\in 2^{<\omega},$ such that for all $s\in 2^{<\omega}, t\in C$ the following holds (see Figure \ref{fig:mapg}):

\begin{align*}
g(R) &= \{p(1)\},  &   &  \\
g(L) &= \{r(X_\emptyset)\}, & &   \\
g(X_{\emptyset})&=\{r\},   & &\\ 
g(X_{1^\frown s}) &= \{q(l(s))\}= \\
&= \{q(r(s))\}, &  & \\
g(X_{0^\frown s}) &= X_{1^\frown s} & g(r(0^\frown s)) &= r(1^\frown s), \\
 & & g(l(0^\frown s)) &= l(1^\frown s), \\
g(1^\frown C)&= R &  g(1^\frown t)&= q(t), \\
g(0^\frown C)&= 1^\frown C &  g(0^\frown t)&= 1^\frown t.  \\
\end{align*}

\begin{figure}[htp]
    \centering
    \includegraphics[width=14cm]{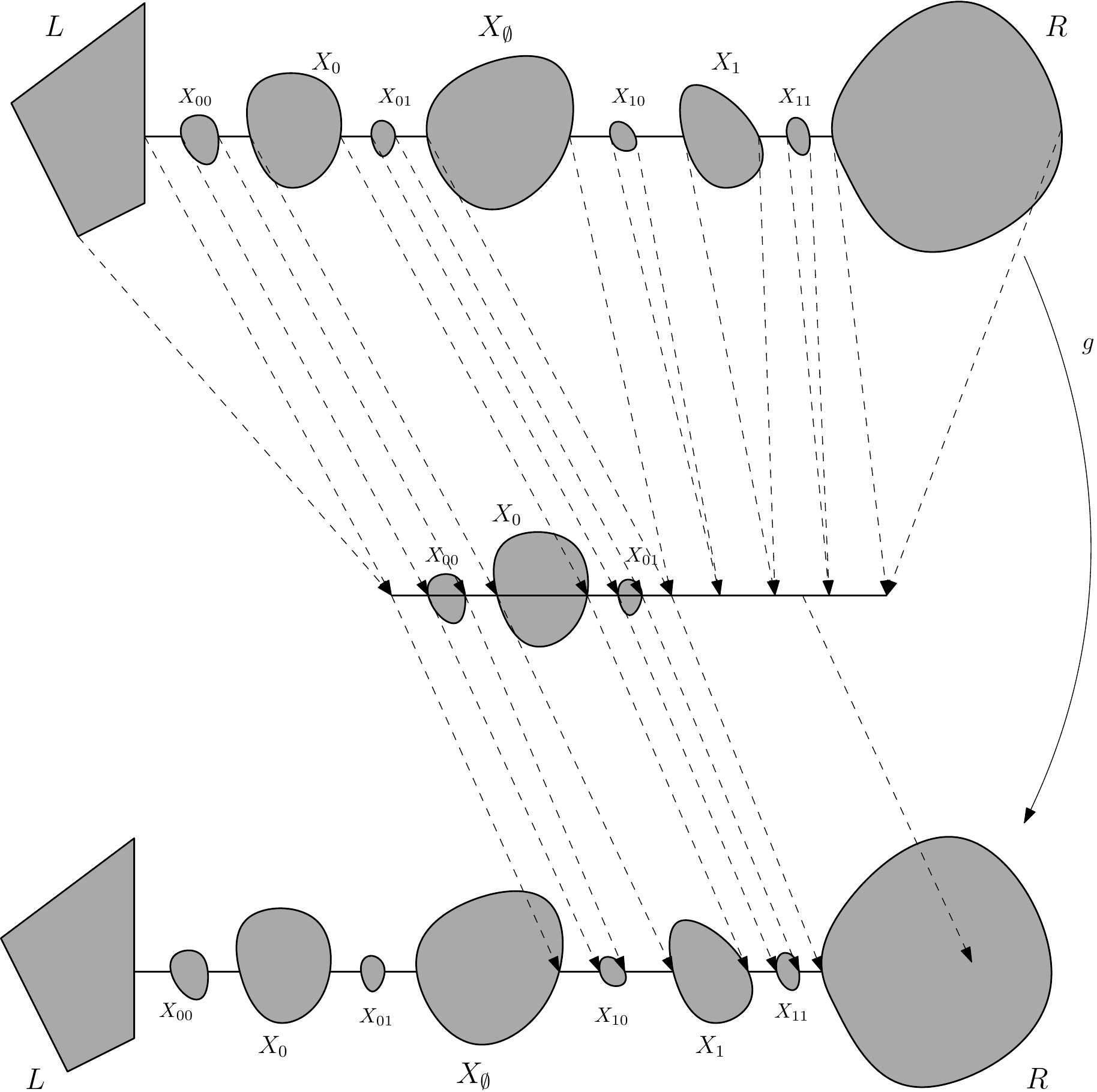}
    \caption{The map $g$}
    \label{fig:mapg}
\end{figure}

It follows directly from the definitions that $f$ and $g$ are well--defined and $X=f(X)\cup g(X)$. Note that $g|_C$ is continuous due to the definition and the basic properties of the Cantor set.
Moreover, $\diam X_s$ tend to zero by Lemma \ref{diameters}, and thus both $f$ and $g$ are continuous by Lemma \ref{continuous}.

Therefore the last thing left to prove is that $\{f,\,g\}$ is topologically contractive. This will follow by Observation \ref{threetotwo} as soon as we justify that $\{f_0,\,f_1,\,g_1\}$ is topologically contractive, where $f_0:= f$, $f_1:= g \circ f$ and $g_1:= g \circ g$. We will prove the latter using Observation \ref{fractalmaps} (with $n=2$ and $m=1$).

To check that the hypotheses of Observation \ref{fractalmaps} are indeed satisfied, first note that since $g$ is constant on $X'$ for every $X' \in \{L,\,R,\,X_\emptyset\} \cup \{X_{1^\frown s},\,s \in 2^{<\omega}\}$ and $g(X_{0^\frown s}) = X_{1^\frown s}$, $g_1 = g \circ g$ is constant on $X'$ for every $X' \in \{L,\,R\} \cup \{X_s,\,s \in 2^{<\omega}\}$. For $f_0 =f$ we obtain $f_0(L) = X_\emptyset$, $f_0(R) = L$ and $f_0(X_s) = X_{0^\frown s^{op}}$ for every $s \in 2^{<\omega}$. Thus $f_1(X_s) = g(f(X_s)) = g(X_{0^\frown s^{op}}) = X_{1^\frown s^{op}}$ for every $s \in 2^{<\omega}$ and $f_1$ is constant on $L$ and on $R$. This immediately gives us that for every $f' \in \bigcup_{i=0}^{\infty} \{f_0,\,f_1\}^i$ the map $g_1\circ f' \circ g_1$ is constant since $g_1(X) = g(g(X)) = R$.

Secondly we have to verify that $\{f_0,\,f_1\}$ is topologically contractive. For fixed $s \in 2^{<\omega}$ let us denote
$$Y_s:= \bigcup \{X_{s^\frown t};\,t \in 2^{<\omega}\} \cup \{s^\frown t;\,t \in 2^{\omega}\}$$
and note that just by the definition of $f_0,\,f_1$ we have that
$f_1(X) = g(f(X))= g(L \cup Y_0 \cup X_\emptyset) =Y_1 \subseteq Y_\emptyset$, $f_0(Y_\emptyset)=f(Y_\emptyset)  = Y_0 \subseteq Y_\emptyset$ and $(f_0)^2 (X) = f_0(L \cup Y_0 \cup X_\emptyset) \subseteq Y_0 \cup X_\emptyset \subseteq  Y_\emptyset$. Thus $f'(X)$ is a subset of $Y_\emptyset$ for every $f' \in \{f_0,\,f_1\}^2$.
	
Finally, note that for every $s \in 2^{<\omega}$ we have that $f_0(Y_s) =  Y_{0^\frown s^{op}}$  and similarly $f_1(Y_s) = Y_{1^\frown s^{op}}$. Thus for every $n \in \N$ and $f' \in \{f_0,\,f_1\}^{n+2}$ we obtain by induction from the preceding that $f'(X) \subseteq Y_s$ for some $s \in 2^{<\omega},\,|s| = n$. Let $\varepsilon >0$, then there exists $n \in \N$ such that for every $s \in 2^{<\omega}$ satisfying $|s| \geq n$ the diameters of $X_s$ and $\{s^\frown t;\,t \in 2^{\omega}\} \subseteq C$ are less than $\varepsilon/3$. Since $X_{s^\frown r}\cap \{s^\frown t;\,t \in 2^{\omega}\} \ne \emptyset$ for every $r \in 2^{<\omega}$, it follows that $\diam Y_s<\varepsilon$ for every $s \in 2^{<\omega}$, $|s| \geq n$.

Thus both hypotheses of Observation \ref{fractalmaps} are satisfied, this gives us that hypotheses of Observation \ref{threetotwo} are satisfied and thus we finally obtain that the maps $f,\,g$ witness that $X$ is a topological fractal.	
\end{proof}

A dendrite is a Peano continuum which does not contain a topological copy of a simple closed curve.

\begin{cor} \label{dendrites}
Every dendrite is a topological fractal with witnessing number two.
\end{cor}

\begin{proof}
A non-degenerate dendrite contains uncountably many cut points by \cite[Theorem 10.8]{nadler}, hence we can use Theorem \ref{thmunccutpoints}.
\end{proof}

\begin{lemma}\label{smallpreimages}
Let $p: X\to S$ be a continuous map of a Peano continuum $X$ onto a circle 
such that only countably many points from $S$ have nondegenerate preimage.
Then for every $\varepsilon>0$ there is a finite set $K\subseteq S$ such that $\diam p^{-1}(L)<\varepsilon$ for every component $L$ of $S\setminus K$. 
\end{lemma}

\begin{proof}
Let $Q$ be a countable dense subset of $S$ which contains the set $\{s\in S: \diam p^{-1}(s)>0\}$. Since $Q$ is countable we can write $Q=\{q_i:i\in\N\}$. Let $K_n=\{q_1,\dots,q_n\}$.
Assume for contradiction that there is $\varepsilon>0$ for which the conclusion does not hold. Then for every $n\in\N$, we can find a component $L_n$ of $S\setminus K_n$ for which $\diam(p^{-1}(L_n))\geq\varepsilon$. Since we have only finitely many choices for each $n\in\N$, we can even suppose using König's lemma that $L_1\supseteq L_2\supseteq\dots$. 
Clearly $\diam(L_i)\to 0$ by the density of $Q$. Thus there exists point $s\in S$ with $\{s\}=\bigcap \cl(L_n)$.

Let $Y_n=\cl(p^{-1}(L_n))$ and $Y=\bigcap Y_n$.
By the assumptions, $Y$ is an intersection of a decreasing chain of continua and hence it is a subcontinuum of $X$. Moreover $\diam(Y)\geq\varepsilon$ since $\diam(Y_n)\geq\varepsilon$ for every $n$. Further,
\[Y= \bigcap \cl (p^{-1}(L_n))\subseteq \bigcap p^{-1}(\cl(L_n))= p^{-1}(\bigcap \cl(L_n))=p^{-1}(s).\]

Thus $p^{-1}(s)$ is nondegenerate and hence $s\notin\bigcap L_n$.
Let $a, b$ be any two distinct points of $Y$. Since $X$ is a Peano continuum we may find disjoint connected open neighborhoods $U, V$ of $a, b$ respectively. Thus $p(U), p(V)$ are connected sets containing $s$.

Let $n$ satisfy $s \notin L_n$. Since $p^{-1}(L_n)$ is dense in $Y_n$, it follows that $U\cap p^{-1}(L_n)\neq\emptyset$ and $V\cap p^{-1}(L_n)\neq\emptyset$ and thus there exists $m\in\N$ for which $L_m\subseteq p(U)\cap p(V)$. Since $S\setminus Q$ is dense in $S$, there is a point $t\in L_m$ with $p^{-1}(t)$ degenerate. Consequently $p^{-1}(t)\subseteq U\cap V$ which implies that $U\cap V\neq\emptyset$, a contradiction.

\end{proof}

\begin{lemma}\label{LemmaMonotoneMapOntoCircle}
Let $S$ be a simple closed curve and $Q$ its countable dense set.
Let $X$ be a Peano continuum with uncountably many local cut points which are not cut points.
Then there is a continuous surjective map $p:X\to S$ for which $p^{-1}(y)$ is a continuum for every $y\in S$ and $p^{-1}(y)$ is nondegenerate if and only if $y\in Q$.
\end{lemma}

\begin{figure}[htp]
    \centering
    \includegraphics[width=8cm]{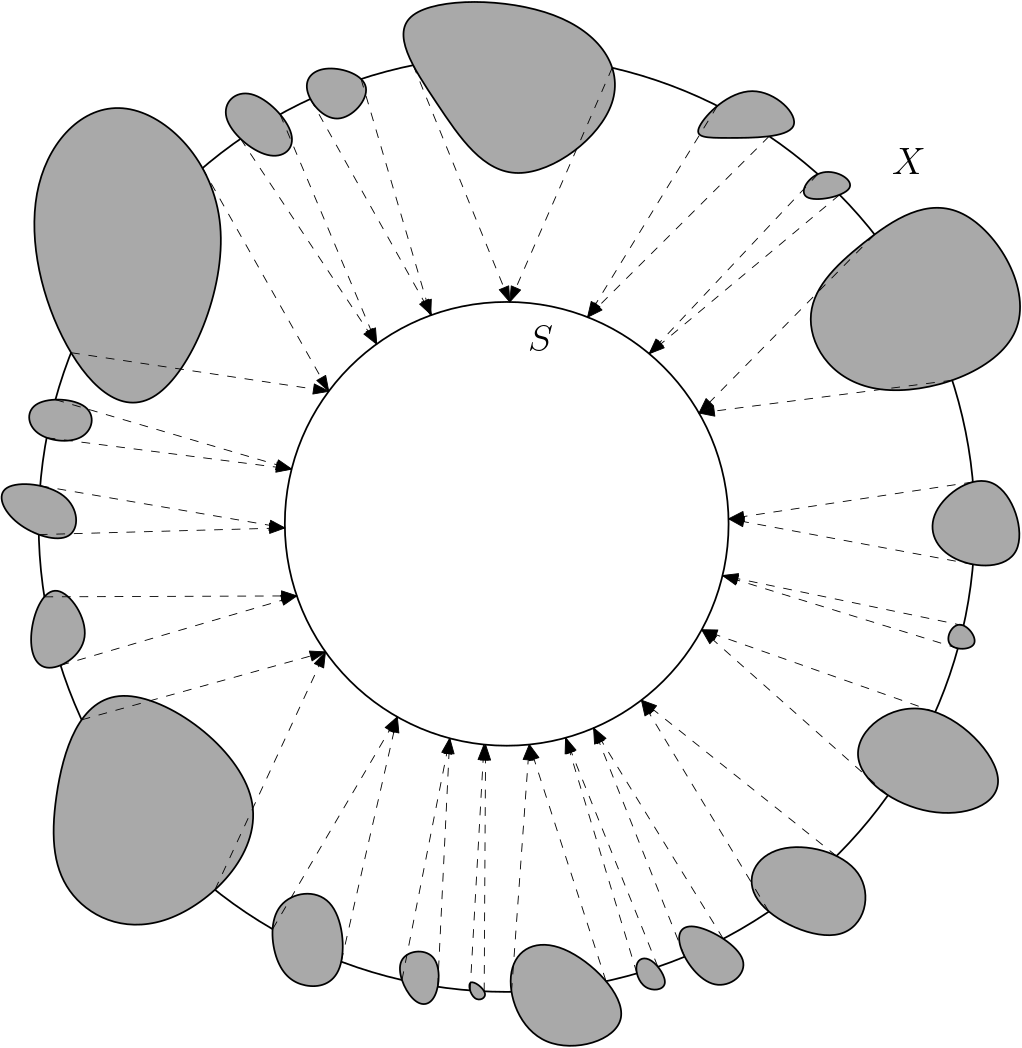}
    \caption{The map $p$}
    \label{fig:mapp}
\end{figure}

\begin{proof}
\textbf{Step 1.}
Let $\mathcal B$ be a countable base of open sets in $X$ formed by nonempty connected sets.
For $P,U,V \in \mathcal B$ with $U\cap V=\emptyset, U\cup V\subseteq P$ let $D(P,U,V) \subseteq X$ be the set of those points $x \in P$ such that there exist $A,B \subseteq P$ disjoint, open and satisfying $P\setminus\{x\}=A\cup B, U\subseteq A,V \subseteq B$. In other words, $D(P,U,V)$ is the set of those local cut points of $X$ that separates $U$ and $V$ in $P$.
It follows easily that 
\[\bigcup \{D(P,U,V);P,U,V \in \mathcal B, U\cup V\subseteq P\}\]
is the set of all local cut points of $X$. 
Thus there exist $P,U,V \in \mathcal B$ such that $D:=D(P,U,V)$ contains uncountably many local cut points which are not cut points.
By Theorem \ref{arcwiseconnected} there is an arc $A\subseteq P$ joining $U$ and $V$. Clearly $D\subseteq A$. Note that that the endpoints of $A$ lie in the same component of $X\setminus C$ since otherwise all points in $D$ would be cut points of $X$. By similar reasons as in Theorem \ref{thmunccutpoints}, there is a topological copy $C$ of the Cantor set in $D$ each point of which is a point of order two.

\textbf{Step 2.}
By \cite[Exercise 8.34]{nadler} only finitely many components of $X\setminus C$ are not contained in $P$. 
Hence we may suppose (possibly by making $C$ even smaller) that every component of $X\setminus C$ except the one which contains endpoints of $A$ is a subset of $P$.
Let $\mathcal P$ be the collection of $\cl (K)$ where $K$ is a component of $X\setminus C$. It follows that $\cl (K)\setminus K$ is a two point set and thus every element of $\mathcal P$ is a Peano continuum by Observation \ref{openconninPeano}.
Moreover the elements of $\mathcal P$ are pairwise disjoint and thus their diameter converges to zero by Lemma \ref{diameters}.

\textbf{Step 3.}
Let $q:X\to T$ be the quotient map of $X$ pushing every continuum $P\in\mathcal P$ to a point (see Figure \ref{fig:mapp}). Note that $T=q(C)$, where $q(s)=q(t)$ for $s,t\in C$ if and only if the subarc in $A$ with endpoints $s$ and $t$ intersects $C$ in $\{s,t\}$ or if it contains the whole set $C$. It follows that the quotient space $T$ is a simple closed curve \cite[Theorem 9.31]{nadler}.
Moreover the set $\{q(P): P\in\mathcal P\}$ is countable and dense in $T$.
By \cite[Theorem 1.6.9]{vanMill}, there is a homeomorphism $r:T\to S$ for which $r(\{q(P):P\in\mathcal P\})=Q$.
Finally let $p=r\circ q$.
We summarize that $p: X\to S$ is a continuous onto map for which $\{p(P):P\in\mathcal P\}=Q$.
\end{proof}

\begin{thm}\label{thmlocalcutpoint}
Let $X$ be a Peano continuum with uncountably many local cut points.
Then $X$ is a topological fractal and the witnessing number is at most three.
\end{thm}

\begin{proof}
If the set of all cut points of $X$ is uncountable, we can use Theorem \ref{thmunccutpoints}. Hence assume that there are uncountably many local cut points which are not cut points and we will find three witnessing maps for $X$ in three steps.

\textbf{Step 1.}
Let us first consider the simple closed curve $S$ and let us find three maps $\alpha, \beta, \gamma$ with some special properties and  witnessing that $S$ is a topological fractal.
Let $a,b,c,d\in S$ be distinct points and $A, B, C, L, R$ be arcs in $S$ such that 
\begin{align*}
S=A\cup B\cup C, A\cap B=\{c\}, 
B\cap C=\{a\} \text{ and } C\cap A=\{b\}, \\ 
d\in C, L\cap R=\{c,d\}, A\subseteq L, B\subseteq R.
\end{align*}
Let $\pi$ be a homeomorphism of $S$ for which 
\begin{align*}
\pi=\pi^{-1}, \pi(c)=c, \pi(d)=d, \pi(a)=b \text{ and thus } \pi(A)=B \text{ and } \pi(L)=R.
\end{align*}
Let $\alpha:S\to A$ be any surjective map for which 
\begin{align*}
\alpha=\alpha\circ\pi, \alpha|_L \text{ is a homeomorphism and } \alpha(c)=b.
\end{align*}
Let $\beta=\pi\circ\alpha$.
Note that $\alpha$ maps the arc $L$ homeomorphically into itself and thus there
is a unique fix point $fix(\alpha)$. Similarly there is a unique fixed point $fix(\beta)$ for $\beta$ and $\pi(fix(\alpha))=fix(\beta)$.
Let $\gamma:S\to C$ be a surjective map with $\gamma=\gamma\circ\pi$, $\gamma^{-1}(b)=C$ and $\gamma^{-1}(a)=\alpha(C)\cup\beta(C)$. Further, choose $\gamma$ so that it maps both $\alpha(A)$ and $\beta(A)$ homeomorphically onto $C$ in such a way that $\gamma(fix(\alpha))=d$ and $\gamma|_E=\pi\circ\gamma\circ\alpha|_E$, where $E\subseteq A$ is the arc with the endpoints $fix(\alpha)$, $b$. 

Finally, let 
\[Q=\{f(a): f\in\{\alpha, \beta, \gamma\}^n, n\geq 0\}.\]
Then $\alpha(Q)\cup\beta(Q)\cup\gamma(Q)=Q$ since $\gamma \alpha (a) = a$.
We can easily observe that $c\notin Q$, since 
\[\bigcup\{f^{-1}(c): f\in \{\alpha, \beta, \gamma\}^n, n\geq 0\}=\{c, d, fix(\alpha), fix(\beta)\}.\]

\begin{claim} \label{claimA}
$\pi(Q)=Q$.
\end{claim} 
It is easy to observe that $\pi(Q\cap A)=Q\cap B$ since $\pi\circ\alpha=\beta$ and $\pi(a) = b= \gamma(a) \in Q$. Thus let $q\in Q\setminus(A\cup B)$, then necessarily $q=\gamma(f(a))$ for some $f\in\{\alpha, \beta, \gamma\}^n, n\geq 0$, such that $f(a)\in E\cup\pi(E)\cup \alpha(E)\cup \beta(E)$. Note that $\alpha(E)\subseteq A$ is the arc with the endpoints $\alpha(a)$, $fix(\alpha)$.
If $f(a)\in E\cup \pi(E)$, then $\pi(q)=\pi(\gamma(f(a)))=\pi(\pi(\gamma(\alpha(f(a)))))= \gamma(\alpha(f(a)))\in Q$. If $f(a)\in \alpha(E)$, resp. $f(a)\in\beta(E)$, then necessarily $f=\alpha \circ f'$, resp. $f=\beta\circ f'$, for some $f'\in \{\alpha,\beta,\gamma\}^{n-1}$, $f'(a)\in E\cup \pi(E)$. In any case, we obtain $q=\gamma(f(a))=\gamma(\alpha(f'(a)))$ since $\gamma\circ\beta = \gamma\circ\pi\circ\alpha=\gamma\circ\alpha$. Hence
\[\pi(q)=\pi(\gamma(\alpha(f'(a))))=\gamma(f'(a)) \in Q.\]

\begin{claim}\label{claimB}
Let $q \in Q\setminus \{a,b\}$ and $\varphi\in\{\alpha, \beta, \gamma\}$. Then $\varphi^{-1}(q)\subseteq Q$.
\end{claim}
Clearly $q\neq c$ as $c\notin Q$.
Consequently, there is unique $\phi\in \{\alpha,\beta,\gamma\}$ with $\phi^{-1}(q)\ne\emptyset$.
Thus we may assume that $q=\varphi(q')$ for some $q'\in Q$ . Note that $\varphi^{-1}(\varphi(s))=\{s,\pi(s)\}$ for every $s\in S$ if $\varphi\in \{\alpha,\beta\}$ and $\gamma^{-1}(\gamma(s))=\{s,\pi(s)\}$ if (and only if) $\gamma(s)\notin \{a,b\}$.
Therefore $\varphi^{-1}(q)=\varphi^{-1}(\varphi(q'))=\{q',\pi(q')\}$ since $q\notin \{a,b\}$.
Hence Claim \ref{claimA} gives $\varphi^{-1}(q)\subseteq Q$. 

Using a suitable metric on $S$ we may suppose that $\alpha, \beta, \gamma$ are contractions (e.g. if $\alpha, \beta,\gamma$ have piecewise constant slope and  $S$ is the circle with length 23 and the lengths of $A, B, C$ are 9, 9, 5 and the lenght of $\alpha(A)$ is 7 and the distance of $fix(\alpha)$ and $\alpha(a)$ is 3, see Figure \ref{fig:circle}). Consequently, $Q$ is a dense subset of $S$.

\begin{figure}[htp]
    \centering
    \includegraphics[width=10cm]{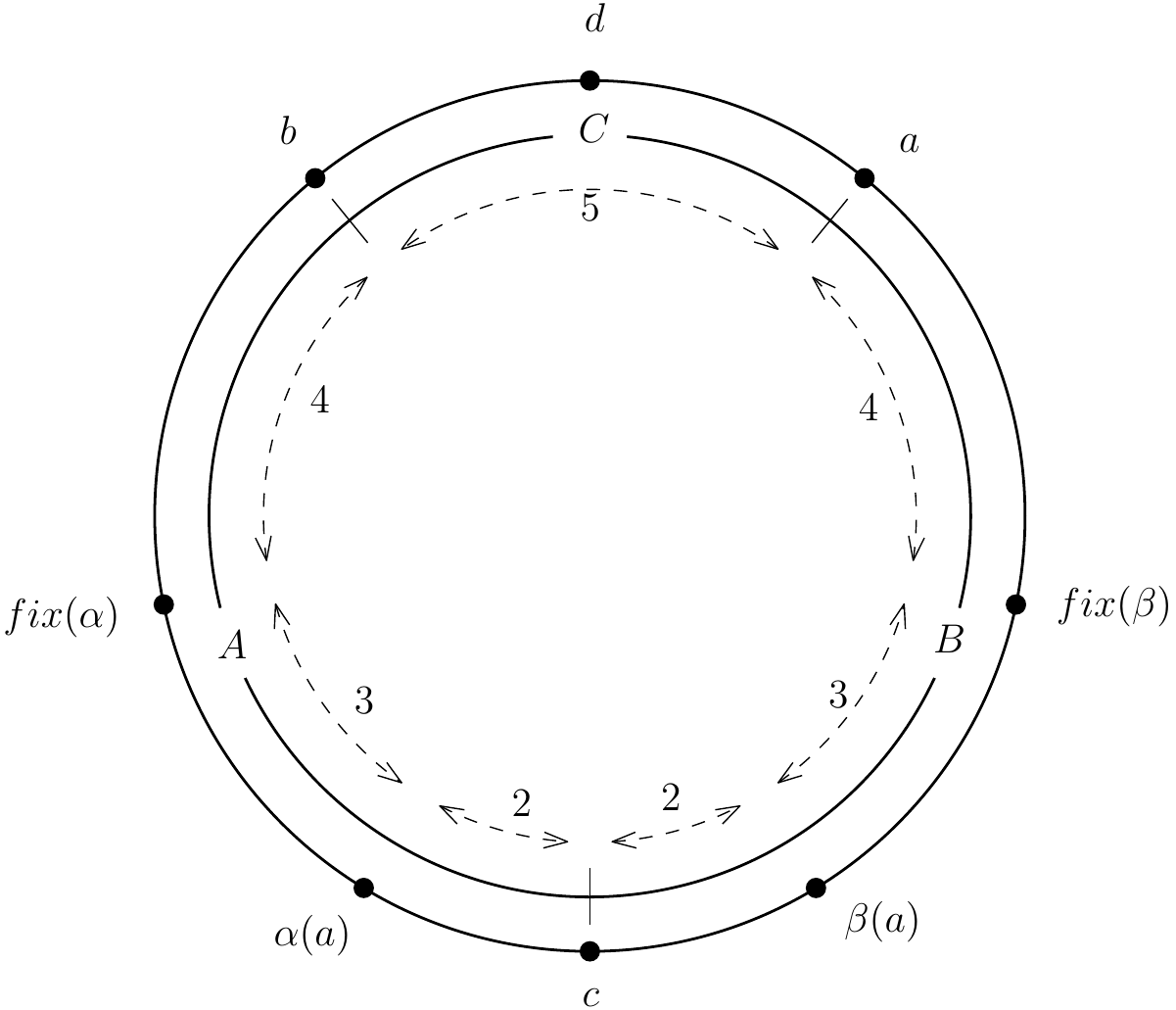}
    \caption{A suitable metric on $S$}
    \label{fig:circle}
\end{figure}

\textbf{Step 2.} 
By Lemma \ref{LemmaMonotoneMapOntoCircle} there exists a surjective map $p:X\to S$ for which $p^{-1}(y)$ is always a continuum and it is nondegenerate if and only if $y\in Q$.
Since by Theorem \ref{coverings} any non-degenerate Peano continuum may be continuously mapped onto any other with a prescription on finitely many points, using Claim \ref{claimB}
we may lift $\alpha, \beta, \gamma$ through the map $p$ to continuous 
maps $\lifted\alpha,\lifted\beta, \lifted\gamma: X\to X$ for which
\begin{align*}
p\circ\lifted\alpha=\alpha\circ p,  \quad
p\circ\lifted\beta=\beta\circ p,   \quad
p\circ\lifted\gamma=\gamma\circ p.
\end{align*}
(This is essentially using the same idea as in the Denjoy example \cite[Section 12.2]{KatokHasselblatt}.)
Moreover we may assume that 
\begin{align*}
\lifted\gamma|_{p^{-1}(C)}=s_b\circ p|_{p^{-1}(C)}, \\
\lifted\gamma|_{p^{-1}(\alpha(C)\cup\beta(C))} = s_a\circ p|_{p^{-1}(\alpha(C)\cup\beta(C))}, \\
\end{align*}
where $s_b:C\to p^{-1}(b)$ and $s_a:\alpha(C)\cup\beta(C)\to p^{-1}(a)$ are suitable continuous surjections (these exists since domains and ranges of $s_a, s_b$ are Peano continua).
Note that $\lifted\gamma|_P$ is constant if $p(P)\in C\cup\alpha(C)\cup\beta(C)$ for $P\in\mathcal P$.

\textbf{Step 3}. It remains to check that the system $\{\lifted{\alpha},\lifted{\beta}, \lifted{\gamma}\}$ is topologically contractive. 
Consider the following maps:
\begin{align*}
\lifted{f_1} &:= \lifted{\alpha}, &&\lifted{f_2} := \lifted{\beta}, 
&&\lifted{f_3} := \lifted{\gamma} \circ \lifted{\alpha} \circ \lifted{\alpha}, \\ \lifted{f_4} &:= \lifted{\gamma} \circ \lifted{\alpha} \circ \lifted{\beta},  &&\lifted{f_5} := \lifted{\gamma} \circ \lifted{\beta} \circ \lifted{\alpha}, &&\lifted{f_6} := \lifted{\gamma} \circ \lifted{\beta} \circ \lifted{\beta},
\\ 
\lifted{g_1} &:= \lifted{\gamma} \circ \lifted{\gamma},
&&\lifted{g_2}:= \lifted{\gamma} \circ \lifted{\alpha} \circ \lifted{\gamma}, &&\lifted{g_3}:= \lifted{\gamma} \circ \lifted{\beta} \circ \lifted{\gamma}.
\end{align*}

We claim that the maps above have the following three key properties. 
\begin{itemize}
    \item[(A)] Every sequence of length $3k$ in $\lifted{\alpha},\lifted{\beta},\lifted{\gamma}$ can be translated into sequence of length $k$ in $\lifted{f_1},\dots,\lifted{f_6},\lifted{g_1},\lifted{g_2},\lifted{g_3}$, starting from the left and possibly leaving out some maps in the end, i.e. on the right.
    \item[(B)] For all $1\leq i,j \leq 3$ and every $\lifted{f} \in \bigcup^\infty_{n=0}\left\{\lifted{f_1},\dots,\lifted{f_6}\right\}^{n}$ the composition $\lifted{g_i} \lifted{f} \lifted{g_j}$ is constant. 
    \item[(C)] $\lifted{f_1},\dots,\lifted{f_6}$ form a topologically contractive system.
\end{itemize}

Once we know that properties (B) and (C) hold, we can use Observation \ref{fractalmaps} to deduce that $\left\{\lifted{f_1},\dots,\lifted{f_6},\lifted{g_1},\dots,\lifted{g_3}\right\}$ is topologically contractive. Consequently, under (B) and (C), (A) entails the topological contractivity of the system $\{\lifted{\alpha},\lifted{\beta},\lifted{\gamma}\}$ by using the same idea as in Observation \ref{threetotwo}. We do not include the full proof of this fact, since it is technical and can be constructed in a manner closely mirroring the proof of Observation \ref{threetotwo}.


It is easy to observe that (A) holds.
To verify (B), fix $1\leq i,j \leq 3$ and $\lifted{f} \in \bigcup^\infty_{n=0}\left\{\lifted{f_1},\dots,\lifted{f_6}\right\}^{n}$. Since $\lifted{g_1}(X) = p^{-1}(b)$ and $\lifted{g_2}(X) \cup \lifted{g_3}(X)= p^{-1}(a)$, we have that $\lifted{g_j}(X) \subseteq P$ for some $P \in \mathcal P$. Since all the maps $\lifted{\alpha}, \lifted{\beta}, \lifted{\gamma}$ send every $P \in \mathcal P$ to another $P' \in \mathcal P$ or to a point, we may suppose that $\lifted{f} \lifted{g_j}(X) \subseteq P'$ for some $P' \in \mathcal P$. Thus it suffices to notice that $\lifted{g_i}$ is constant on $P'$ or in other words, that $\lifted{\gamma}$ is constant on $\lifted{\gamma}(P'), \lifted{\alpha}\lifted{\gamma}(P')$ and $\lifted{\beta}\lifted{\gamma}(P')$. This is immediate as these are either degenerate or satisfy $\lifted{\gamma}(P'), \lifted{\alpha}\lifted{\gamma}(P')$, $\lifted{\beta}\lifted{\gamma}(P') \in \mathcal P$, $p(\lifted{\gamma}(P')), p(\lifted{\alpha}\lifted{\gamma}(P'))$, $p(\lifted{\beta}\lifted{\gamma}(P'))\in C\cup \alpha(C)\cup \beta(C)$.


We proceed to prove (C). Let $Y\subseteq X$ be connected, we say that $Y$ has \emph{degenerate ends} if $Y=\cl (p^{-1}(\int p(Y)))$. This basically means that the intersections of $Y$ with the preimages of endpoints of $p(Y)$ are singletons. Observe that for every $Y\subseteq X$ connected and $1\leq i \leq 6$, if $Y$ has degenerate ends, then $\lifted f_i(Y)$ has degenerate ends as well.
Since $X$ has degenerate ends, we obtain by induction that $\lifted{f}(X)$ has degenerate ends for any $\lifted{f} \in \bigcup^\infty_{n=0}\{\lifted{f_1},\dots,\lifted{f_6}\}^n$.

Let $\varepsilon > 0$. 
By Lemma \ref{smallpreimages} there is a finite set $K\subseteq S$ such that for every component $L$ of $S\setminus K$ we have $\diam p^{-1}(L)<\varepsilon$.
Since $Q$ is dense in $S$ we can even suppose that $K\subseteq Q$.
Since $p$ is monotone the set $p^{-1}(L)$ is a component of $X\setminus p^{-1}(K)$ for any such $L$.
Note that every point of $Q$ is contained in the interior of only finitely many sets of the form $f(S)$, $f\in \bigcup^\infty_{n=0}\{\alpha, \beta, \gamma\}^{n}$.
Hence there is $n\in\mathbb N$ such that $K\cap \int (f(S))=\emptyset$ for every $m\geq n$ and $f\in \{\alpha, \beta, \gamma\}^m$.

Denote by $f_i: S\to S$ maps obtained from $\alpha,\beta,\gamma$ in the same way as $\lifted{f_i}: X\to X$ are obtained from $\lifted{\alpha},\lifted{\beta},\lifted{\gamma}$. Thus $f_1=\alpha$, $f_2=\beta$, $f_3=\gamma\circ\alpha\circ\alpha=f_4=f_5=f_6$ and moreover $f_i\circ p=p\circ \lifted{f_i}$.
Fix arbitrary $\lifted{f} \in \{\lifted{f_1},\dots,\lifted{f_6}\}^n$ and
consider the corresponding $f\in \{f_1,\dots,f_6\}^n$ with $f\circ p=p\circ\lifted f$.
Then $f\in\{\alpha,\beta,\gamma\}^m$ for some $m\geq n$ and hence $K\cap \int (f(S))=\emptyset$. 
The set $p^{-1}(\int (f(S)))$ is connected ($p$ is monotone and thus preimages of connected sets are connected) and hence it is a subset of some component of $X\setminus p^{-1}(K)$.
Consequently, $\diam p^{-1}(\int p(\lifted f(X)))=\diam p^{-1}(\int (f(p(X))))=\diam p^{-1}(\int (f(S)))<\varepsilon$.
Thus $\diam \lifted{f}(X)=\diam \cl (p^{-1}(\int p(\lifted f(X))))\leq \varepsilon$.
\end{proof}

\section{Self-regenerating fractals}

The notion of a self-regenerating fractal was introduced in \cite{regenerating}. Here are the definition and the main result from that article.

\begin{defn}
A compact metrizable space $A$ is called a self-regenerating fractal if for every
nonempty, open set $U\subset A$ there exists a finite family $\F$ of continuous self-maps of $A$, constant outside $U$, which witness that $A$ is a topological fractal.
\end{defn}

\begin{thm}\cite[Theorem 4.8]{regenerating}\label{thmnowak}
If a Peano continuum $X$ has $A \subset X$ a self-regenerating fractal with nonempty
interior, then $X$ is a topological fractal itself.
\end{thm}

Since the arc is a self-regenerating fractal \cite{regenerating}, the above result generalizes the fact that every Peano continuum with a free arc is a topological fractal. Yet Theorems \ref{thmunccutpoints} and \ref{thmlocalcutpoint} from the previous section, though being themselves generalizations of the fact that every Peano continuum with a free arc is a topological fractal, are independent of Theorem \ref{thmnowak}.

With Theorem \ref{thmnowak} in mind, M. Nowak formulated in her article two problems \cite[Problem 4.9, 4.10]{regenerating}. Slightly reformulated they follow.

\begin{question}\label{questionA}
Does every connected topological fractal contain a
self-regenerating fractal with nonempty interior?
\end{question}

\begin{question}\label{questionB}
Does every Peano continuum contain a self-regenerating fractal with nonempty interior?
\end{question}

Question \ref{questionA} asks whether the converse of Theorem \ref{thmnowak} is true, while the positive answer to Question \ref{questionB} together with Theorem \ref{thmnowak} would give us immediately positive answer to Hata's question. Using the main result proven in the previous section, we can show that these two questions of M. Nowak are actually equivalent. In fact, it should be clear that Question \ref{questionB} is stronger than Question \ref{questionA}, so we are only left to justify the converse. We will need the following lemma in the proof.

\begin{lemma} \label{regeneratingsub}
Let $Y$ be a Peano continuum and $r: Y \to A$ a monotone retraction onto $A\subseteq Y$ with finite boundary. Let $R\subseteq Y$ be a self-regenerating fractal and $R'$ a connected component of $A\cap R$ with nonempty interior in $Y$.
Then $R'$ is a self-regenerating fractal.
\end{lemma}

\begin{proof}
We may assume that $A$ is nondegenerate. Since $Y$ is Peano, there are $W\subsetneq V\subseteq \int (R')$ nonempty, open and connected. Since $R$ is self-regenerating, there exist finitely many continuous selfmaps of $R$ constant outside $W$ whose images cover $R$. Since $V$ is connected, each of these images is connected, and thus $R$ is a union of finitely many connected sets. Denote by $V'$ the connected component of $V$ in $R$. Then $V'$ is clopen in $R$.

We want that $r(V') = R'$. First, $R'\subseteq V'$ since $V\subseteq R'\subseteq R$ and $R'$ is connected. Thus, since $R'\subseteq A$, we obtain $R' = r(R') \subseteq r(V')$.
For the opposite inclusion, observe first that for $a\in \int (A)$ we have $r^{-1}(a)\subseteq \{a\}\cup (Y\setminus A)$ just by $r|_A = id_A$. Thus the monotonicity of $r$ enforces $r^{-1}(a)= \{a\}$ for every $a \in \int (A)$, in other words $r(Y\setminus A) \subseteq \bd(A)$. Thus $V'\cap A\subseteq r(V') \subseteq (V'\cap A) \cup \bd(A)$, where $r(V')$ is connected, $V'\cap A$ is compact and $\bd(A)$ is finite. This implies that $r(V') = V'\cap A\subseteq R$. Finally, this gives us that $V\subseteq r(V') \subseteq A\cap R$ and thus $r(V')$, being a connected set, is a subset of $R'$.

Let $U' \subseteq R'$ be an arbitrary nonempty open set in $R'$. Then there exists $U \subseteq Y$ open such that $U'=U\cap R'$. Its subset $U\cap \int (A) \cap V'$ is nonempty and open in $R$ (it differs by only finitely many points). Moreover, we may suppose that it is a proper subset of $V'$.
Thus there exists $\F$ a set of continuous self--maps of $R$ constant outside $U\cap \int (A)\cap V'$ witnessing that $R$ is a topological fractal. Since for every $f \in \F$ either $f(V')\subseteq V'$ or $f(V') \cap V' = \emptyset$, after possibly leaving out some maps we may assume that $\bigcup \{f(V');\,f \in \F\} = V'$. Note that $\F':= \{f|_{R'};\,f \in \F\}$ is still topologically contractive and the images cover $V'$.
For any $f \in \F'$ let $f':= r\circ f$ and note that $f':R'\to R'$.

Let $\F'':= \{f';\,f \in \F'\}$. Then $\F''$ consists of continuous self--maps of $R'$, $\F''$ is topologically contractive and $R' = \bigcup \{f(R');\,f \in \F''\}$. Thus $\F''$ is a family of maps constant outside $U'$ witnessing that $R'$ is a topological fractal.
	 
\end{proof}

\begin{thm} \label{thmqA=qB}
Question \ref{questionA} has positive answer if and only if Question \ref{questionB} has positive answer.
In fact, even the following assertions are equivalent:
\begin{itemize}
\item Every Peano continuum which is a topological fractal contains a
self-regenerating fractal with nonempty interior.
\item Every Peano continuum contains a Peano self-regenerating fractal with nonempty interior. 
\end{itemize}
\end{thm}

\begin{proof}
The backward implication is trivial, so we are left to prove the forward one. Suppose that every Peano continuum which is a topological fractal contains a self-regenerating fractal as a subset with nonempty interior and let $X$ be an arbitrary Peano continuum. We may assume that $X$ is nondegenerate.

The idea of what we are going to do is, roughly speaking, to attach one copy of $X$ in between every pair of consecutive points in the Cantor set (consecutive with respect to the natural linear order on the Cantor set), the “smaller" copy the closer the points are, and in this way we obtain a connected topological fractal. However, we will carry out the construction via inverse limits.

Fix distinct points $a,b \in X$. For $s\in 2^{<\omega}$ let $X_s$ be a topological copy of $X$ and let $c_{s^\frown 0}, c_{s^\frown 1}\in X_s$ be the corresponding copies of points $a, b\in X$, respectively. Let us consider the strict linear order $\prec$ on $2^{<\omega}$ given by
\[s^{\frown}0^{\frown}u \prec s \prec s^{\frown}1^{\frown}v\]
for every $s, u, v\in2^{<\omega}$.
Let $Y_n$ be the quotient space of $\bigoplus\{ X_s: s\in 2^{<n}\}$, where points $c_{s^\frown 1}\in X_s$ and $c_{t^\frown 0}\in X_t$ are identified if $s\prec t$, $s, t\in2^{<n}$, and there is no $r\in 2^{<n}$ with $s\prec r\prec t$.
Let $f_n: Y_{n+1}\to Y_n$ be a map which maps $X_s\subseteq Y_{n+1}$ onto $X_s\subseteq Y_n$ by the identity if $|s|<n$ and which maps every point of $X_s\subseteq Y_{n+1}$ to the point $c_s\in Y_n$ if $|s|=n$.
Let $Y$ be the inverse limit of the inverse system $(Y_n, f_n)_{n=1}^\infty$, i.e. 
\[Y=\{(y_n)_{n=0}^\infty: y_n\in Y_n, f_n(y_{n+1})=y_n\}\]
and denote by $\pi_n:Y\to Y_n$ the natural projection.
Clearly, every $f_n$ is continuous and monotone and every $Y_n$ is a Peano continuum. Consequently, $Y$ is a Peano continuum by \cite[Exercise 8.47]{nadler}.

The points $z_s:=(c_{s|n})_{n\in\omega}$, $s\in 2^{\omega}$, all lie in $Y$. Moreover, they are pairwise distinct since in between any two points of $2^{n}$ there is a point in $2^{<n}$. Let us argue that if $s\neq \periodic 0, \periodic 1$, then $z_s$ is a cut point of $Y$. 
Consider the set
\[C:= \bigcap_{n\in\omega} \pi_n^{-1}\left(\bigcup\left\{ X_t: t\in2^n, t \prec s|_{n} \text{ or } t=s|_{n}\right\}\right).\]
Clearly, $C$ is closed and $z_{\periodic 0}\in C$.
We can define $D$ similarly, but with the opposite order of $\prec$. Then $D$ is closed, $z_{\periodic 1}\in D$ and $C\cup D=Y$.
Moreover 
\[C\cap D=\bigcap \pi_n^{-1}(X_{s|n}) = \bigcap \pi_{n-1}^{-1}(c_{s|n}) =  \{(c_{s|n})_{n\in\omega}\}=\{z_s\}.\]
Thus $z_s$ is a cut point of $Y$.

Thus there are uncountably many cut points in $Y$, hence by Theorem \ref{thmunccutpoints} the space $Y$ is a topological fractal and therefore, by our assumption, there is a self-regenerating fractal $R\subseteq Y$ with nonempty interior.
	
Thus there exists $n\in \N$ and a nonempty open set $U\subseteq Y_n$ such that $\pi_n^{-1}(U)\subseteq R$. We may suppose that there is $s\in 2^{<\omega}$ such that $U\subseteq X_s\setminus \{c_{s^\frown 0}, c_{s^\frown 1}\}$. Since $X_s$ is a Peano continuum we may also suppose that $U$ is connected. Let

\[A:= \{(x_{k})_{k\in\omega};\forall k\geq n: x_k = x_{k+1} \in X_s\},\]
\[V:= \{(x_{k})_{k\in\omega};\forall k\geq n: x_k = x_{k+1} \in U\}=\pi_n^{-1}(U).\]
Then $A\subseteq Y$ is homeomorphic to $X_s$ and $V\subseteq A \cap R$ is homeomorphic to $U$. In particular, $V$ is nonempty, connected and open. 
The set $X_s$ is a retract of every $Y_m$, $m\geq n$, in a canonical way, hence by
\cite[Exercise 2.22]{nadler} also $A$ is a retract of $Y$. Moreover it is easy to observe that the retraction is monotone as preimages of only two points of $A$ are nondegenerate.
The boundary of $A$ is $\{z_{s^\frown 0\periodic 1}, z_{s^\frown 1\periodic 0}\}$ and thus finite.

Let $R'$ be the connected component of $A\cap R$ which contains $V$.
By $V\subseteq R'$ we obtain $\int( R') \ne \emptyset$ and therefore $R'$ is a self-regenerating fractal by Lemma \ref{regeneratingsub}.
Since $R'$ is a connected topological fractal it is also a Peano continuum.
Thus $X_s$ contains a Peano self-regenerating fractal with nonempty interior.
Finally, since $X_s$ is homeomorphic to $X$, the same holds for $X$ and this concludes the proof.

\end{proof}

\begin{cor}
If Question \ref{questionA} has positive answer then every Peano continuum is a topological fractal.
\end{cor}

\begin{proof}
If Question \ref{questionA} has positive answer, then Question \ref{questionB} has positive answer by Theorem \ref{thmqA=qB}. Thus, by Theorem \ref{thmnowak}, every Peano continuum is a topological fractal.
\end{proof}

\section{Final remarks and questions}

When studying topological fractals, it is natural to ask what is the least upper bound of wittnesing numbers for some classes of spaces. Clearly, the number $2$ is the least possible witnessing number for any nondegenerate space, which makes Theorem \ref{thmunccutpoints} optimal. However, it is not clear whether the number $3$ for Peano continua with uncountably many local points is optimal. In fact, the optimality question is non--trivial even for the simplest possible space --- the circle:

\begin{question}
What is the witnessing number of the simple closed curve?
\end{question}

Naturally, we can extend the preceding question into two natural directions: graphs and spheres. For (compact metrizable) topological graphs, the witnessing number is always either two or three. 

\begin{question}
Characterize those topological graphs with witnessing number two.
\end{question}

By simple geometric reasons we can argue that there is a family of $n+2$ maps witnessing that the $n$-dimensional sphere is a topological fractal. However, what is the witnessing number is unclear.

\begin{question}
What is the witnessing number of the $n$-dimensional sphere?
\end{question}

\begin{rem}
If $S=S^n$ is the unit $n$-dimensional sphere (with a metric inherited from $\mathbb R^{n+1}$) and $S=f_1(S)\cup\dots f_k(S)$ for weak contractions $f_1, \dots, f_k$, then $k\geq n+2$.
This is because if for contradiction $k\leq n+1$ then by the Lusternik–Schnirelmann theorem, 
at least one of the sets $f_i(S)$ contains a pair $(x,-x)$ of antipodal points for some $i\in\{1,\dots, k\}$. Then there are points $a, b\in S$ for which $f_i(a)=x, f_i(b)=-x$. Since $f_i$ is a weak contraction we get $2=\|x-(-x)\|=\|f_i(a)-f_i(b)\| < \|a-b\|\leq 2$ which is a contradiction.

However, note that the preceding argumentation relies on the fact that we have a fixed (nice) metric on $S$, so we can not directly conclude anything about the witnessing number of the $n$-dimensional sphere.
\end{rem}

\begin{rem}
Note that $[0,1]^n$ is a topological fractal with only two witnessing maps. Indeed, let
\[ f(x_1,\dots,x_n):= (x_2/2,x_3,\dots,x_n,x_1),\]
\[ g(x_1,\dots,x_n):= (1/2+x_2/2,x_3,\dots,x_n,x_1).\]
Then the family $\{f,g\}$ witness that $[0,1]^n$ is a topological fractal.
\end{rem}


In contrast to the last remark, the following question is still open.
\begin{question}
Is the Hilbert cube a topological fractal?
\end{question}

Dumitru examined topological fractals with two injective witnessing maps and observed that there are infinitely many such dendrites \cite[Remark 2.2]{Dumitru}. As we know by 
Corollary \ref{dendrites}, every dendrite is a topological fractal with witnessing number two, however the witnessing maps constructed in the proof of Theorem \ref{thmunccutpoints} are not injective.

Rim-finite continua (i.e. a continua having a base whose elements have finite boundaries) need not to have uncountably many local cut-points (e.g. the Sierpiński triangle). However, rim-finite continua are locally connected. Thus the following question seems to be natural in view of the fact that dendrites are rim-finite.

\begin{question}
Is every rim-finite continuum a topological fractal?
\end{question}

An \emph{absolute retract} is any retract of the Hilbert cube. Since every dendrite is embeddable into the plane \cite{CharatoniksDendrites}, the following result is a strengthening of Corollary \ref{dendrites}.

\begin{prop}
Every planar absolute retract is a topological fractal.
\end{prop}

\begin{proof}
Let $X\subseteq \R^2$ be an absolute retract. We distinguish two cases.
If $X$ is nowhere dense then $X$ is one-dimensional by \cite[Theorem 1.8.11]{EngelkingDimension} and thus it is a dendrite by \cite[Theorem 1.2]{CharatoniksDendrites}. Consequently, $X$ is a topological fractal by Corollary \ref{dendrites}.
If $X$ is not nowhere dense in $\R^2$ then $X$ contains a topological copy of $[0,1]^2$ with nonempty interior. Since $[0,1]^2$ is easily seen to be a self-regenerating fractal, we conclude by Theorem \ref{thmnowak} that $X$ is a topological fractal.
\end{proof}

Finally, the interesting results in the theory of topological fractals do not necessarily restrict to connected spaces only.
For example, among zero--dimensional spaces even complete characterization was provided in \cite{zerodimensional}.

\bibliographystyle{alpha}

\newcommand{\etalchar}[1]{$^{#1}$}

\end{document}